\documentclass{article}
\usepackage{graphicx,amsmath,amssymb,amsthm}
\usepackage[english]{babel}
\usepackage[sorting=none]{biblatex}
\usepackage[a4paper, left = 1in,
right = 1in,
top =1in,
bottom=1in]{geometry}
\usepackage{comment}
\title{Rado's Graph has no Quantum Symmetry}
\author{Husam Ismaeel \thanks{ Email address: {\texttt{2826284i@student.gla.ac.uk}}}}
\date{}
\newtheorem{lemma}{Lemma}[section]
\newtheorem{theorem}{Theorem}[section]

\newtheorem{definition}{Definition}[section]
\renewenvironment{abstract}
 {\quotation\noindent\textbf{Abstract. }}
 {\endquotation}
\begin{document}

\maketitle
\begin{abstract}
    We prove that Rado's graph admits no quantum symmetries. 
\end{abstract}

\section{Introduction}
Quantum automorphisms of finite simple graphs were introduced in \cite{Bichon2003}, and explored initially by Banica and Bichon \cite{BanicaBichon2007},\cite{Banica2005}. They have several interesting connections to quantum groups and graph theory, see for example the remarkable results in \cite{Mancinska2020}.
The notion of quantum automorphism was extended to arbitrary simple graphs by Voigt \cite{Voigt2023}, where the question of whether  Rado's graph admitted quantum symmetry was posed. Interest in Rado's graph grew after Erdős and Rényi showed that there exists a graph R with the property that it is isomorphic, with probability 1, to a countable graph where each pair of vertices is connected independently with probability  $\frac{1}{2}$ \cite{Erdos1963}. Shortly after, Richard Rado explicitly constructed the graph R \cite{Rado1964}, hence the name.

This work was done as part of a summer project funded by the University of Glasgow. The author is extremely grateful to Christian Voigt for his thorough guidance, frequent discussions, and encouragement to work on this topic. 
\section{Preliminaries}
In this section, we define quantum automorphisms of simple graphs as done in \cite{Voigt2023}, and introduce Rado's graph and some of its useful properties. 
\begin{definition}
    A quantum permutation of a set X is a pair $\sigma = (H,u)$, where $H$ is a Hilbert space and $u = (u_{xy})_{x,y \in X}$ is a family of projections in $B(H)$ such that
    \[
    \sum_{y \in X} u_{xy} = \sum_{y \in X} u_{yx} = 1,
    \]
   where convergence is taken in the strong operator topology. Also, for $(x,y) \neq (a,b)$, $u_{xy}$ and $u_{ab}$ are pairwise orthogonal projections if $x = a$  or $y=b$.
\end{definition}
Note that this definition encompasses classical permutations of  sets $X$, where $u$ is the permutation matrix, and the Hilbert space $H$ is $\mathbb{C}$.
\begin{definition}
    A quantum automorphism $\sigma = (H, u)$ of a simple graph $X = (V_X,A)$, where $V_X$ is the set of vertices, and $A$ is the adjacency matrix, is a quantum permutation of $V_X$ such that 
    \begin{equation} \label{eq:autconstraints}
    u_{x_1y_1}u_{x_2y_2} = 0,
    \end{equation}
    if $A_{x_1x_2} \neq A_{y_1y_2}$.
\end{definition}
Again, this definition generalises automorphisms of simple graphs $X$, 
since if $u$ was the classical permutation matrix, the requirement above enforces $u$ to commute with the adjacency matrix $A$.  A graph $X$ is said to have \textit{quantum symmetry} if it has a quantum automorphism $(H,u)$ where not all the elements of $u$ commute.

There are several ways to describe Rado's graph $R$. One straightforward description is to have $R = (V_R, A)$  where the vertices $V_R$ are the prime numbers congruent 1 mod 4, and $A_{pq} =1$ if and only if $p$ is a quadratic residue mod $q$. It is also the unique simple countable graph such that for any disjoint finite sets $U$ and $V$ of $V_R$, there exists a vertex $i$ such that $A_{iu} = 1$ for all $u \in U$, and $A_{iv} = 0$ for $v \in V$. See \cite{Cameron1997} for a proof.

\section{Main Result}
Throughout we let $ R = (V_R, A) $ be the Rado graph, with $V_R$ being the set of vertices and $A$ its adjacency matrix. 
\begin{lemma} \label{lem:finitesum}
Let $ (H,u) $ be a quantum automorphism of $ R $, let $ P_0, P_1, Q_0, Q_1 $ be finite subsets of $V_R$ such that $ (P_0 \cup Q_0) \cap(P_1 \cup Q_1)  =  \varnothing $, and let $ y, t \in V_R $ be two distinct vertices. Consider the projections
\begin{align*}
p_0 &= \sum_{x \in P_0} u_{xy}, \quad  p_1 = \sum_{x \in P_1} u_{xy}, \\
q_0 &= \sum_{x \in Q_0} u_{xt}, \quad q_1 = \sum_{x \in Q_1} u_{xt}.
\end{align*}
Then for any $ v \in p_1(H) $ there exist vectors $ v_0 \in p_0^\perp(H) \cap q_0^\perp(H), v_1 \in p_0^\perp(H) \cap q_1^\perp(H) $ such that $ v = v_0 + v_1 $ and $ \langle v_0, v_1 \rangle = 0 $. 
\end{lemma}

\begin{proof}
For $\alpha, \beta \in \{0,1\}$ we define
\[
V_{\alpha \beta} = \{j \in V_R \mid A_{yj} \neq \alpha \text{ and } A_{tj} \neq \beta \}.
\]
Then the sets $ V_{\alpha \beta} $ are mutually disjoint, and their union is $ V_R $. 

Since $ P_0 \cup Q_0 $ and $ P_1 \cup Q_1 $ are disjoint, it follows from the properties of Rado's graph that there exists $ i \in V_R $ such that $ A_{xi} = 0 $ for  all $ x \in P_0 \cup Q_0 $ and $ A_{xi} = 1 $ for all $ x \in P_1 \cup Q_1 $. Hence, for all $ \alpha, \beta \in \{0,1\} $, we have 
\begin{alignat*}{2}
u_{xy}u_{ij} &= 0 = u_{ij}u_{xy} \qquad &&\text{ if } x \in P_\alpha \text{ and } j \in V_{\alpha \beta}, \\ 
u_{xt} u_{ij} &= 0 = u_{ij}u_{xt} &&\text{ if } x \in Q_\beta \text{ and } j \in V_{\alpha \beta},
\end{alignat*}
which implies 
\begin{alignat*}{2}
p_\alpha u_{ij} &= 0 = u_{ij} p_\alpha \qquad &&\text{ if }  j \in V_{\alpha \beta}, \\ 
q_\beta u_{ij} &= 0 = u_{ij} q_\beta &&\text{ if } j \in V_{\alpha \beta},
\end{alignat*}
by the definition of the projections $ p_\alpha, q_\beta $. This shows 
\begin{equation*} \label{eq:imagecontainment}
u_{ij}(H) \subseteq p^\perp_{\alpha}(H) \cap q^\perp_\beta(H).
\end{equation*}
for all $ j \in V_{\alpha \beta} $. 

Now let $ v = p_1 v \in p_1(H) $ and define $ v_0 = \sum_{j \in V_{00}} u_{ij} v $ 
and $ v_1 = \sum_{j \in V_{01}} u_{ij}v $. For $ j \in V_{1\beta} $ with $ \beta \in \{0,1\} $ 
we have $ u_{ij} v = u_{ij} p_1 v = 0 $ by our above considerations, and thus
\[
v = \sum_{j \in V_R} u_{ij} v = \sum_{\alpha, \beta \in \{0,1\}} \sum_{j \in V_{\alpha \beta}} u_{ij} v = \sum_{j \in V_{00}} u_{ij} v + \sum_{j \in V_{01}} u_{ij} v = v_0 + v_1.
\]
Since $ u_{ij} v \in p_{0}^\perp (H) \cap q_\beta^\perp(H) $ for $ j \in V_{0\beta} $ we also obtain $ v_\beta \in p_{0}^\perp (H) \cap q_\beta^\perp(H)$ for $ \beta \in \{0,1\} $, and finally, 
\(
 \langle v_0, v_1 \rangle = 0, 
\)
since projections in the same row of $ u $ are orthogonal.  
\end{proof}
\begin{theorem}
Let $ (H,u) $ be a quantum automorphism of $ R $ and let $ x, y, s, t \in V_R$ be arbitrary. Then $u_{xy}$ and $u_{st}$ commute. 
\end{theorem}

\begin{proof}
If $ x = s $ or $ y = t $ then $ u_{xy} $ and $ u_{st} $ clearly commute because $ u $ is a magic unitary. 
We shall therefore assume that $ x \neq s $ and $ y \neq t $ in the sequel. 

Let $ P_1 = \{x\}, Q_1 = \{x,s\} $, and $ P_0 = \{z\} $ for $ z \in V_R \setminus Q_1$, and let $ Q_0 $ be any finite subset of $ V_R \setminus Q_1$. Then $ (P_0 \cup Q_0) \cap (P_1 \cup Q_1) = \varnothing $, and we define
\begin{alignat*}{2}
p_0 &= \sum_{r \in P_0} u_{ry} = u_{zy}, & p_1 &= \sum_{r \in P_1} u_{ry} = u_{xy}, \\
q_0 &= \sum_{r \in Q_0} u_{rt}, \qquad \qquad & q_1 &= \sum_{r \in Q_1} u_{rt} =  u_{xt} + u_{st}. 
\end{alignat*}
Note that $ p_1 q_1 = u_{xy} u_{st} $ and $ q_1 p_1 = u_{st} u_{xy} $, so that it is enough to show that $ q_1 p_1 $ is self-adjoint. For this, in turn, it suffices to check $ p_1^\perp q_1 p_1 = 0 $ since this implies $ q_1 p_1 = (p_1 + p_1^\perp) q_1 p_1 = p_1 q_1 p_1 $. 

As a preliminary step, we show $ p_0 q_1 p_1 = 0 $. 
Applying Lemma \ref{lem:finitesum} to the sets $ P_0, Q_0, P_1, Q_1 $, it follows that for a unit vector $v \in p_1(H) $ there exists orthogonal vectors $ v_0 \in p_{0}^\perp(H) \cap q_0^\perp(H), v_1 \in p_{0}^\perp(H) \cap q_1^\perp(H) $ such that $ v = v_0 + v_1 $. Since $ v_1 \in q_1^\perp(H) $ we get 
$ q_1 v = q_1 v_{0}  $. If we set $ w = v_0 - q_1v $, then this implies $ q_1 w = q_1 v_0 - q_1 v = 0 $, so that $ w \in q_1^\perp(H) $. 

Next note that $ q_1(H) $ is contained in $ q_0^\perp(H) $ because $ Q_0 $ and $ Q_1 $ are disjoint. Since $ v_{0} $ is contained in $ q_0^\perp(H) $ as well we conclude that $ w= v_0 - q_1v \in q_0^\perp(H) $. Thus, if we let 
\begin{equation*}
q = \sum_{r \in Q_0 \cup Q_1} u_{rt} = q_0 + q_1,
\end{equation*}
then we get $ w \in q_0^\perp(H) \cap q_1^\perp(H) = q^\perp(H) $. 

In addition, recalling that $ w = v_0 - q_1 v = v_0 - q_1 v_0 = q_1^\perp v_0 $, and that $ v_0, v_1 $ are orthogonal, we have 
\[
\| w \| \leq \|v_{0} \| \leq \|v_0 + v_1 \| = \| v \| = 1.
\]
Using the Cauchy-Schwarz inequality we therefore get 
\[
 |\langle w, p_0 q_1v \rangle | = |\langle q^\perp w, p_0 q_1 v \rangle | = |\langle w, q^\perp p_0 q_1v \rangle | \leq \| w \| \| q^\perp p_0 q_1 v \| \leq \| q^\perp p_0q_1 v \|.
\]
Since $v_0 \in p_0^\perp(H)$, we also have
\[
| \langle w, p_0q_1v \rangle | = | \langle v_0  - q_1v, p_0q_1v \rangle | = | \langle -p_0 q_1v, p_0q_1v \rangle | = | \langle p_0 q_1v, p_0q_1v \rangle | = \| p_0q_1v \|^2.
\]
Combining these formulas gives
\begin{equation*} 
\| p_0q_1v \|^2 \leq \| q^\perp p_0q_1v \|.
\end{equation*}
Now, by choosing the finite set $ Q_0 $ defining $ q_0 $ sufficiently large, the term $ \|q^\perp p_0q_1v \| $ can be made arbitrarily small. We therefore conclude $ p_0 q_1v = 0 $, and since $v \in p_1(H)$ was arbitrary this shows $p_0q_1p_1 = 0$. 

Thus, for each $z \in V_{R} \setminus Q_1$ we have $u_{zy} q_1p_1 = 0 $, which implies 
\[
\bigg(\sum_{r \in V_R \setminus \{ x,s\}} u_{ry} \bigg)q_1p_1 = 0.
\]
Finally, we note that
\(
u_{sy} q_1p_1 = u_{sy}u_{xt}u_{xy}  + u_{sy}u_{st}u_{xy}  = 0
\)
since $ s \neq x $ and $ y \neq t $, so that 
\[
\bigg(\sum_{r \in V_R \setminus\{x\}} u_{ry} \bigg) q_1p_1 = p_1^\perp q_1p_1 = 0
\]
as required. This finishes the proof.
\end{proof}
\printbibliography
\end{document}